\documentclass[a4paper]{amsart}

\usepackage{amssymb}
\usepackage[hidelinks]{hyperref}
\usepackage{todonotes}

\usepackage{xcolor}

\theoremstyle{plain}
\newtheorem{theorem}{Theorem}[section]

\newtheorem{corollary}[theorem]{Corollary}
\newtheorem{proposition}[theorem]{Proposition}

\theoremstyle{definition}
\newtheorem{definition}[theorem]{Definition}

\newtheorem{example}[theorem]{Example}
\newtheorem{problem}{Problem}
\newtheorem{question}{Question}

\theoremstyle{remark}
\newtheorem{remark}{Remark}[section]

\newcommand{\bR}{\mathbb{R}}
\newcommand{\R}{\bR}

\newcommand{\bN}{\mathbb{N}}
\newcommand{\N}{\bN}

\newcommand{\cA}{\mathcal{A}}

\newcommand{\cI}{\mathcal{I}}
\newcommand{\I}{\cI}

\newcommand{\cP}{\mathcal{P}}

\newcommand{\continuum}{\mathfrak{c}}

\newcommand{\fin}{\mathrm{Fin}}
\renewcommand{\subset}{\subseteq}

\DeclareMathOperator{\ran}{ran}
\DeclareMathOperator{\FS}{FS}
\DeclareMathOperator{\nosnik}{supp}

\begin{document}

\title{The ideal test for the divergence of a series}

\author[R.~Filip\'{o}w]{Rafa\l{} Filip\'{o}w}
\address[R.~Filip\'{o}w]{Institute of Mathematics\\ Faculty of Mathematics, Physics and Informatics\\ University of Gda\'{n}sk\\ ul. Wita Stwosza 57\\ 80-308 Gda\'{n}sk\\ Poland}
\email{Rafal.Filipow@ug.edu.pl}
\urladdr{http://mat.ug.edu.pl/~rfilipow}

\author[A.~Kwela]{Adam Kwela}
\address[A.~Kwela]{Institute of Mathematics\\ Faculty of Mathematics\\ Physics and Informatics\\ University of Gda\'{n}sk\\ ul.~Wita  Stwosza 57\\ 80-308 Gda\'{n}sk\\ Poland}
\email{Adam.Kwela@ug.edu.pl}
\urladdr{https://mat.ug.edu.pl/~akwela}

\author[J.~Tryba]{Jacek Tryba}
\address[J.~Tryba]{Institute of Mathematics\\ Faculty of Mathematics, Physics and Informatics\\ University of Gda\'{n}sk\\ ul. Wita Stwosza 57\\ 80-308 Gda\'{n}sk\\ Poland}
\email{Jacek.Tryba@ug.edu.pl}

\date{\today}

\subjclass[2020]{Primary: 40A05, 46B87, 40A35 Secondary: 15A03, 46B45.}

 
\keywords{Lineability, spaceability, algebrability, convergent series, convergence test, Olivier's theorem, ideal, ideal convergence, summable ideal, Borel ideal.}


\begin{abstract}
We generalize the classical Olivier's theorem which says that for any convergent series $\sum_n a_n$ with positive nonincreasing real terms the sequence $(n a_n)$ tends to zero. Our results encompass many known generalizations of Olivier's theorem and give some new instances. The generalizations are done in two directions: we either drop the monotonicity assumption completely or we relax it to the monotonicity on a large set of indices. In both cases, the convergence of $(na_n)$ is replaced by ideal convergence.  

In the second part of the paper, we examine families of sequences for which the assertions of our generalizations of Olivier's theorem fail. Here, we are interested in finding large linear and  algebraic substructures  in these families.
\end{abstract}

\maketitle


\section{Introduction}

A basic test for divergence of an infinite series says that for any sequence $(a_n)$ of reals we have $$\left| \sum_{n=1}^\infty a_n\right| <\infty \implies  \lim_{n\to\infty} a_n=0.$$ 
The following theorem says more about the speed of convergence to zero of the terms of a convergent series with positive and nonincreasing terms.
\begin{theorem}[Olivier~\cite{MR1577632}]
\label{thm:Olivier}
For every nonincreasing sequence $(a_n)$ of positive reals we have 
$$ \sum_{n=1}^\infty a_n<\infty \implies  \lim_{n\to\infty} na_n=0.$$ 
\end{theorem}

A nonempty family $\I\subseteq\cP(\N)$ of subsets of $\N$ is an \emph{ideal on $\N$} if it is closed under taking subsets and finite unions of its elements, $\N\notin\I$ and $\I$ contains all finite subsets of $\N$. By $\fin$ we denote the ideal of all finite subsets of $\N$.
The family 
$$\I_{1/n}=\left\{A\subseteq\N: \sum_{n\in A}\frac{1}{n}<\infty\right\}$$ 
is an ideal on $\N$ that is called \emph{the summable ideal}.

For an ideal $\I$ on  $\N$ and a sequence  $(a_n)$ of reals, we write 
$$\I-\lim a_n=L$$ 
if
$\{n\in\N: |x_n-L|\geq \varepsilon\}\in\I$ for every $\varepsilon>0$. Obviously  $\fin-\lim a_n$ coincides with the ordinary limit of a sequence $(a_n)$. (More information on ideal convergence can be found e.g.~in \cite{MR1844385}.)

The following theorem shows that using the ideal convergence, we can  drop the  monotonicity assumption in Olivier's theorem.
\begin{theorem}[{\v{S}}al{\'a}t-Toma~\cite{MR2031274}]
\label{thm:salat-toma}
Let $\I$ be an ideal on $\N$. The following conditions are equivalent. 
\begin{enumerate}
    \item For every  sequence $(a_n)$ of positive reals we have 
$$\sum_{n=1}^\infty a_n<\infty \implies \I-\lim n a_n=0.$$
\item The ideal $\I$ extends the summable ideal:
$$\I_{1/n}\subseteq  \I.$$
\end{enumerate}
\end{theorem}

\smallskip

For an ideal $\I$ on $\N$, we write $\I^*=\{\N\setminus A: A\in\I\}$ and call it the \emph{filter dual to $\I$}.
For a sequence  $(a_n)$ of reals, we write 
$$\I^*-\lim a_n=L$$ 
if
there exists $F\in \I^*$ such that the subsequence $(a_n)_{n\in F}$ has the limit $L$ i.e.
$$\exists F\in \I^*\, \forall \varepsilon>0\, \exists k \, \forall n\in F\, \left(k\leq n \implies  |a_n - L |<\varepsilon\right).$$
It is known (\cite[Proposition~3.2 and Theorem~3.2]{MR1844385})
that $\I^*-\lim a_n=L$ implies $\I-\lim a_n=L$, whereas  the reversed implication holds if and only if $\I$ is a P-ideal (an ideal $\I$ is a \emph{P-ideal} if for each countable family $\cA\subseteq \I$ there exists $B\in \I^*$ such that $B \cap A$ is finite for each $A\in \cA$).

In Section~\ref{subsec:general1}, we prove some generalizations (see Theorems~\ref{thm:general1} and  \ref{thm:general2}) of the above mentioned theorem of \v{S}al\'{a}t and Toma. Our results encompass other generalizations of Theorem~\ref{thm:salat-toma} considered in the literature (see \cite{MR4317877,MR4317877-Correction} and \cite{MR4327685}). Among others, we show (Corollary~\ref{cor:Salat-Toma}) that $\I$-convergence can be replaced by a much stronger condition of $\I^*$-convergence in Theorem~\ref{thm:salat-toma}. 

The results of Section~\ref{subsec:general1} utilizes summable ideals (defined by Mazur~\cite[p.~105]{MR1124539}) which generalize the summable ideal $\I_{1/n}$ and are defined in the following manner: 
for every divergent series $\sum_{n=1}^\infty d_n=\infty$ with nonegative terms we define a \emph{summable ideal generated by a sequence $(d_n)$} by
$$\I_{(d_n)} = \left\{A\subset\N: \sum_{n\in A}d_n<\infty\right\}.$$
All summable ideals are  P-ideals (see e.g.~\cite[Example~1.2.3(c)]{MR1711328}).

For an ideal $\I$ on $\N$, we write $\I^+=\{ A: A\notin\I\}$ and call it the \emph{coideal of $\I$}.
For a  sequence $(a_n)$ of reals we write 
$$\I^+-\lim a_n=L$$ 
if 
there exists $A\in \I^+$ such that the subsequence $(a_n)_{n\in A}$ has the limit $L$ i.e.
$$\exists A\in \I^+\, \forall \varepsilon>0\, \exists k \, \forall n\in A\, \left(k\leq n \implies  |a_n - L |<\varepsilon\right).$$
Note that $\I^+$-limit of a sequence does not have to be unique. This kind of limit was considered in  \cite{MR2320288}, where the authors proved among others that for a large class of ideals (e.g.~for the summable ideal $\I_{1/n}$) every bounded sequence of reals has $\I^+$-limit.  

Obviously $\I^*-\lim a_n=L$ implies $\I^+-\lim a_n=L$, and the reversed implication holds if and only if $\I$ is a maximal ideal (an ideal $\I$ is maximal if $\I^+=\I^*$). In general, there is no relation between $\I$-convergence and $\I^+$-convergence, but  
$\I^+-\lim a_n=L$ implies $\I-\lim a_n=L$ if and only if $\I$ is a maximal ideal, whereas  
$\I-\lim a_n=L$ implies $\I^+-\lim a_n=L$ if and only if $\I$ is a  weak P-ideal (an ideal $\I$ is a \emph{weak P-ideal} if for each countable family $\cA\subseteq \I$ there exists $B\in \I^+$ such that $B\cap A$ is finite for each $A\in \cA$).

In Section~\ref{subsec:general3}, we prove (see Theorem~\ref{thm:general3}) similar results  as in Section~\ref{subsec:general1}, but with the aid of $\I^+$-convergence which is independent of $\I$-convergence in general.

We say that a sequence $(a_n)$ of reals  is \emph{$\I^*$-nonincreasing} if there exists $F\in \I^*$ such that the subsequence $(a_n)_{n\in F}$ is nonincreasing i.e.
$$\exists F\in \I^*\, \forall n,k\in F\, \left(n\leq k \implies  a_n\geq a_k\right).$$

In \cite{MR3576795}, the authors have been weakening  the assumption on monotonicity in Olivier's theorem instead  of dropping it entirely. Among others, they posted the following  problem.
\begin{problem}[{Faisant-Grekos-Mi\v{s}\'{i}k~\cite{MR3576795}}]
Characterize ideals $\I$ with the property that for every $\I^*$-nonincreasing sequence $(a_n)$ of positive reals we have
$$\sum_{n=1}^\infty a_n<\infty \implies \I^*-\lim na_n =0.$$ 
\end{problem}
One can see that Olivier's theorem (Theorem~\ref{thm:Olivier}) says that the ideal $\I=\fin$ has this  property.
On the other hand, in \cite{MR3576795}, the authors construct an ideal $\I$ without this property.

In Section~\ref{sec:general4-monotone}, we solve the above problem by providing  (Theorems~\ref{thm:star-ideal}
and \ref{thm:star-plus}) some characterizations of the above mentioned property and other properties of similar flavour.

In all of the above  mentioned results  there is an assumption that the considered series has positive terms, but we can weaken this assumption to consider also absolutely convergent series with arbitrary terms (as $\I-\lim n|a_n| =0$ implies $\I-\lim na_n=0$, and similarly for other types of convergences).
However, the alternating harmonic series 
$\sum_{n=1}^\infty (-1)^n/n$ 
shows that both Olivier's theorem and Theorem~\ref{thm:salat-toma} fail (as the sequence $(-1)^n$ is not $\I$-convergent to zero for any ideal $\I$) if we allow the series to be conditionally convergent in these theorems.
On the other hand, it is known (Kronecker \cite{Kronecker}, see also \cite[Theorem~82.3, p.~129]{Knopp}) that a version of Olivier's theorem for arbitrary series holds if one replaces ordinary convergence by Ces\'{a}ro convergence (also known as Ces\'{a}ro summation or the Ces\'{a}ro mean).

\begin{theorem}[Kronecker \cite{Kronecker}]
If a series $ \sum_{n=1}^\infty a_n$ is convergent, then the sequence $(na_n)$ is Ces\'{a}ro convergent to zero i.e. 
$$\lim_{n\to\infty} \frac{a_1+2a_2+\dots+na_n}{n}=0.$$ 
\end{theorem}

\bigskip

For more than a decade, the research on finding linear subsets of nonlinear sets in vector spaces  (the trend nowadays known as \emph{lineability} or \emph{spaceability}) is gathering more and more mathematicians. Below we provide the notions we will use in the last part of our paper (for more on the subject see e.g.~the  book \cite{MR3445906} or the survey~\cite{MR3119823}).  

Let  $X$ be a  Banach algebra and    $\kappa$ be a cardinal number.
We say that a subset $Y\subseteq X$ is 
\begin{enumerate}
    \item 
 \emph{$\kappa$-lineable} if $Y\cup\{0\}$ contains a vector subspace of dimension $\kappa$;
 

\item 
 \emph{$\kappa$-spaceable}  if $Y\cup\{0\}$ contains a closed vector subspace of dimension $\kappa$;

\item 
 \emph{$\kappa$-algebrable} if $Y\cup\{0\}$ contains a $\kappa$-generated subalgebra;
 

\item 
 \emph{strongly $\kappa$-algebrable} if $Y\cup\{0\}$ 
contains a $\kappa$-generated subalgebra which is a free algebra;

\item \emph{lineable} if it is $\kappa$-lineable for some  infinite $\kappa$ (and similarly for other notions defined above).
\end{enumerate}

In \cite{MR4317877}, the authors consider the Banach algebra $\ell_1$ (of all real sequences $(a_n)$ with absolutely convergent series $\sum_{n=1}^\infty a_n$ equipped with the norm  $\lVert (a_n)\rVert = \sum_{n=1}^\infty |a_n|$ and coordinatewise addition and multiplication) and examine the lineability-like notions of the set of those sequences for which the assertion of Olivier's theorem fails, namely they examine the set: 
$$AOS = \{(a_n)\in \ell_1:(na_n) \text{ is not convergent to zero}\}.$$
Among others, they proved that $AOS$ is strongly $\continuum$-algebrable, $\continuum$-lineable and spaceable.

In Section~\ref{sec:algebraic-structures}, we examine the lineability-like notions of the following  sets: 
\begin{equation*}
    \begin{split}
AOS(\I) & = \{(a_n)\in \ell_1:(na_n) \text{ is not $\I$-convergent to zero}\}, \\
AOS(\I^*) & = \{(a_n)\in \ell_1:(na_n) \text{ is not $\I^*$-convergent to zero}\}, \\
AOS(\I^+) & = \{(a_n)\in \ell_1:(na_n) \text{ is not $\I^+$-convergent to zero}\}
    \end{split}
\end{equation*}
for an arbitrary ideal $\I$ on $\N$.

Since $\I^*$-convergence implies both $\I$-convergence and $\I^+$-convergence,  $AOS(\I)\subseteq AOS(\I^*)$ and $AOS(\I^+)\subseteq AOS(\I^*)$. However,  in general,  these inclusions do not reverse (see Proposition~\ref{prop:relations-between-AOSs}).

In Section~\ref{subsec:lineability}, we prove that a  necessary and sufficient condition for $AOS(\I)$, $AOS(\I^*)$ and $AOS(\I^+)$ to be  $\continuum$-lineable is that these families are nonempty (see Theorem~\ref{the:lineability}).

In Section~\ref{subsec:spaceability}, we describe some classes of ideals for which 
a  necessary and sufficient condition for $AOS(\I)$ and $AOS(\I^*)$  to be  spaceable is that these families are nonempty (see
Theorems~\ref{thm:spaceability-1} and \ref{thm:spaceability-2}).
An example of such a class of ideals is the family of all Borel ideals (see Theorem~\ref{thm:spaceability-1-Borel}).
However, we do not know if this condition works for every ideal (see Question~\ref{q:spaceable-iff}). 
Moreover, we do not know any conditions under which $AOS(\I^+)$ is spaceable (see Question~\ref{q:nontrivial-non-spaceable-ideal}).

In Section~\ref{subsec:algebrability}, we describe some classes of ideals for which $AOS(\I)$, $AOS(\I^*)$ and $AOS(\I^+)$  are strong $\continuum$-algebrable (see Theorem~\ref{thm:strong-algebrability}).


\section{The ideal test for the divergence of an infinite series}


\subsection{$\I$ and $\I^*$ tests}
\label{subsec:general1}

\begin{theorem}\label{thm:general1}
Let $\I$ be an ideal on $\N$.
Let $g:(0,\infty)\to(0,\infty)$ be a strictly increasing function such that
$$\lim_{x\to 0^+} \frac{g(x)}{x^\gamma}=M$$
 for some positive constants  $\gamma,M\in(0,\infty)$.
Let   $(b_n)$ and  $(c_n)$ be sequences of positive reals such that 
$$\lim_{n\to\infty} c_n=\infty.$$
Then the following conditions are equivalent.
\begin{enumerate}
    \item For every sequence $(a_n)$ with  $a_n\in\ran(g)$ we have \label{thm:general1:I-star-convergence}
$$ \sum_{n=1}^{\infty} b_n a_n<\infty \implies   \I^*-\lim c_n  g^{-1}(a_n)=0. $$

\item For every sequence $(a_n)$ with  $a_n\in\ran(g)$ we have\label{thm:general1:I-convergence}
$$ \sum_{n=1}^{\infty} b_n  a_n<\infty \implies  \I-\lim c_n g^{-1}(a_n)=0.$$

\item 
The ideal $\I$ extends the summable ideal generated by the sequence $(b_n g(1/c_n))$:
$$ \I_{\left(b_n g(1/c_n)\right)} \subseteq \I.$$\label{thm:general1:ideals}
\end{enumerate}
\end{theorem}
\begin{proof}

$(\ref{thm:general1:I-star-convergence})\implies(\ref{thm:general1:I-convergence})$
It follows from the fact that $\I^*$-convergence implies $\I$-convergence.

\smallskip

$(\ref{thm:general1:I-convergence})\implies(\ref{thm:general1:ideals})$
Suppose that there exists $A\in \I_{\left(b_n  g(1/c_n)\right)}\setminus\I$. For any $n\in\N$, take as $d_n$ some element of $\ran(g)$ with $d_n\leq \frac{1}{n^2 b_n}$. We can find such an element by the assumption $\lim_{x\to 0^+} g(x)=0$.
We define
$$a_n=\left\{\begin{array}{ll}
g(1/c_n) & \textrm{ for } n\in A,\\
d_n & \textrm{ for other } n. \end{array}\right. $$

We can see that $\sum_{n=1}^{\infty} b_n a_n<\infty$. Indeed, observe that
$$\sum_{n=1}^{\infty} b_n a_n\leq \sum_{n\in A}b_n g(1/c_n) + \sum_{n\not\in A} \frac{b_n}{n^2 b_n}.$$
Since $A\in \I_{\left(b_n g(1/c_n)\right)}$, the first sum is finite and the second sum is finite because the series $\sum \frac{1}{n^2}$ is convergent. 

It follows from the assumption that  $\I-\lim c_n g^{-1}(a_n)=0$. On the other hand, for every $n\in A$ we have 
$$ c_n g^{-1}(a_n)=c_n  g^{-1}\left( g(1/c_n) \right)=c_n (1/c_n)=1.$$
Since $A\not\in\I$, we obtain $\I-\lim c_n g^{-1}(a_n)\not=0$, a contradiction.

\smallskip

$(\ref{thm:general1:ideals})\implies(\ref{thm:general1:I-convergence})$
 Let $\sum_{n=1}^{\infty}b_n a_n<\infty$. We need to show that  $\I-\lim c_n g^{-1}(a_n)=0$.
Take $\varepsilon>0$. We consider the set 
$$A=\{n\in\N:c_n g^{-1}(a_n)\geq\varepsilon \}=\{n\in\N: a_n\geq g(\varepsilon/c_n) \}=\{n\in\N: b_n a_n\geq b_n g(\varepsilon/c_n) \}. $$ 

In order to prove that $A\in\I$, we only need to show that $\sum_{n\in A} b_n g(1/c_n)<\infty$.
First we notice that $\sum_{n\in A} b_n g(\varepsilon/c_n)<\infty$ since 
$$\sum_{n\in A} b_n g(\varepsilon/c_n)\leq \sum_{n\in A}b_n a_n\leq \sum_{n=1}^{\infty}b_n a_n<\infty.$$
Now, we would like to prove that convergence of the series $\sum_{n\in A} b_n g(\varepsilon/c_n)$ implies convergence of the series $\sum_{n\in A} b_n g(1/c_n)$.

Notice that since $\lim_{n\to\infty} c_n=\infty$, we have $\lim_{n\to\infty} 1/c_n=0$. Therefore 
$$\lim_{n\to\infty} \frac{g(1/c_n)}{(1/c_n)^\gamma}=M \textrm { and } \lim_{n\to\infty} \frac{g(\varepsilon/c_n)}{(\varepsilon/c_n)^\gamma}=M. $$
It follows that 
$$\lim_{n\to\infty} \frac{g(1/c_n)}{g(\varepsilon/c_n)}\cdot \frac{(\varepsilon/c_n)^\gamma}{(1/c_n)^\gamma}=1,$$ thus
$$\lim_{n\to\infty} \frac{g(1/c_n)}{g(\varepsilon/c_n)}=\varepsilon^{-\gamma}\in(0,\infty). $$
Because of that, $\sum_{n\in A} b_n  g(\varepsilon/c_n)<\infty$ is equivalent to  $\sum_{n\in A} b_n g(1/c_n)<\infty$.

\smallskip

$(\ref{thm:general1:ideals})\implies(\ref{thm:general1:I-star-convergence})$
Since
  $(\ref{thm:general1:ideals})\implies(\ref{thm:general1:I-convergence})$
  and $\I_{\left(b_n g(1/c_n)\right)} \subseteq \I_{\left(b_n g(1/c_n)\right)}$, 
the convergence of the series $\sum_{n=1}^{\infty} b_n a_n$ implies
$\I_{\left(b_n  g(1/c_n)\right)}-\lim  c_n  g^{-1}(a_n)=0$. 
Since $\I_{\left(b_n g(1/c_n)\right)}$ is a P-ideal, we obtain $\I_{\left(b_n  g(1/c_n)\right)}^*-\lim  c_n  g^{-1}(a_n)=0$. By the assumption  $\I_{\left(b_n g(1/c_n)\right)}\subseteq\I$,  we have $\I_{\left(b_n g(1/c_n)\right)}^*\subseteq\I^*$, thus $\I^*-\lim c_n g^{-1}(a_n)=0$.
\end{proof}

\begin{theorem}\label{thm:general2}
Let $\I$ be an ideal on $\N$.
Let $g:(0,\infty)\to(0,\infty)$ be a strictly increasing function such that
$$\lim_{x\to 0^+} g(x)=0 \text{\ \ and \ \ } 
\forall \varepsilon>0 \, \exists M \, \forall x>0\, \left(\frac{g(x)}{g(\varepsilon x)}\leq M\right).$$
Let   $(b_n)$ and  $(c_n)$ be sequences of positive reals.
Then the following conditions are equivalent.
\begin{enumerate}
\item For every sequence $(a_n)$ with  $a_n\in\ran(g)$ we have\label{thm:general2:I-star-convergence}
$$\sum_{n=1}^{\infty} b_n a_n<\infty \implies   \I^*-\lim c_n g^{-1}(a_n)=0.  $$

\item For every sequence $(a_n)$ with  $a_n\in\ran(g)$ we have\label{thm:general2:I-convergence}
$$ \sum_{n=1}^{\infty} b_n a_n<\infty \implies   \I-\lim c_n g^{-1}(a_n)=0.$$    

\item 
The ideal $\I$ extends the summable ideal generated by the sequence $(b_n g(1/c_n))$:
$$ \I_{\left(b_n g(1/c_n)\right)} \subseteq \I.$$\label{thm:general2:ideals}

\end{enumerate}
\end{theorem}
\begin{proof}
$(\ref{thm:general2:I-star-convergence})\implies(\ref{thm:general2:I-convergence})$
It follows from the fact that $\I^*$-convergence implies $\I$-convergence.

\smallskip

$(\ref{thm:general2:I-convergence})\implies (\ref{thm:general2:ideals})$ 
Suppose that there exists $A\in \I_{\left(b_n g(1/c_n)\right)}\setminus\I$. For any $n\in\N$, take as $d_n$ some element of $\ran(g)$ with $d_n\leq \frac{1}{n^2 b_n}$. We can find such an element because  $\lim_{x\to 0^+} g(x)=0$.
 We define
$$a_n=\left\{\begin{array}{ll}
g(1/c_n) & \textrm{ for } n\in A,\\
d_n & \textrm{ for other } n. \end{array}\right. $$

We can see that $\sum_{n=1}^{\infty} b_n a_n<\infty$. Indeed, observe that
$$\sum_{n=1}^{\infty} b_n a_n\leq\sum_{n\in A}b_n g(1/c_n) + \sum_{n\not\in A} \frac{b_n}{n^2 b_n}.$$
Since $A\in \I_{\left(b_n g(1/c_n)\right)}$, the first sum is finite and the second sum is finite because the series $\sum \frac{1}{n^2}$ is convergent. 

It follows from the assumption that  $\I-\lim c_n  g^{-1}(a_n)=0$. On the other hand, for every $n\in A$ we have 
$$ c_n  g^{-1}(a_n)=c_n g^{-1}\left( g(1/c_n) \right)=c_n (1/c_n)=1.$$
Since $A\not\in\I$, we obtain $\I-\lim c_n g^{-1}(a_n)\not=0$, a contradiction.

\smallskip

$(\ref{thm:general2:ideals})\implies (\ref{thm:general2:I-convergence})$ 
Let $\sum_{n=1}^{\infty}b_n a_n<\infty$. We need to show that  $\I-\lim c_n g^{-1}(a_n)=0$.
Take $\varepsilon>0$. We consider the set 
$$A=\{n\in\N:c_n g^{-1}(a_n)\geq\varepsilon \}=\{n\in\N: a_n\geq g(\varepsilon/c_n) \}=\{n\in\N: b_n a_n\geq b_n  g(\varepsilon/c_n) \}. $$ 

In order to prove that $A\in\I$, we only need to show that $\sum_{n\in A} b_n g(1/c_n)<\infty$.
First we notice that $\sum_{n\in A} b_n g(\varepsilon/c_n)<\infty$ since 
$$\sum_{n\in A} b_n g(\varepsilon/c_n)\leq \sum_{n\in A}b_n a_n\leq \sum_{n=1}^{\infty}b_n a_n<\infty.$$
Now, observe that there exists $M$ such that for all $n\in\N$ we have 
$$\frac{g(1/c_n)}{g(\varepsilon/c_n)}\leq M, $$
thus 
$$\sum_{n\in A} b_n g(1/c_n)\leq M \sum_{n\in A} b_n g(\varepsilon/c_n)<\infty.$$

\smallskip

$(\ref{thm:general2:ideals})\implies(\ref{thm:general2:I-star-convergence})$
Since
  $(\ref{thm:general2:ideals})\implies(\ref{thm:general2:I-convergence})$
  and $\I_{\left(b_n g(1/c_n)\right)} \subseteq \I_{\left(b_n g(1/c_n)\right)}$, 
the convergence of the series $\sum_{n=1}^{\infty} b_n a_n$ implies
$\I_{\left(b_n g(1/c_n)\right)}-\lim  c_n g^{-1}(a_n)=0$. 
Since $\I_{\left(b_n g(1/c_n)\right)}$ is a P-ideal, we obtain $\I_{\left(b_n g(1/c_n)\right)}^*-\lim  c_n g^{-1}(a_n)=0$. By the assumption  $\I_{\left(b_n g(1/c_n)\right)}\subseteq\I$, we have $\I_{\left(b_n g(1/c_n)\right)}^*\subseteq\I^*$, thus $\I^*-\lim c_n g^{-1}(a_n)=0$.
\end{proof}

\begin{remark}\
\label{non-negative}
\begin{enumerate}
    \item 
Notice that both Theorems~\ref{thm:general1} and \ref{thm:general2} would still be true if we add $0$ to the domain and codomain of $g$ and require that $g(0)=0$. There is even no need to change any of the proofs, and then we can strengtheen these theorems by requiring the sequence  $(a_n)$ to be non-negative instead of positive.
\label{non-negative:non-negative-instead-positive} 

\item 
Note that Theorem~\ref{thm:general2}
does not imply Theorem~\ref{thm:general1}.
Indeed, the function $g(x)=e^x-1$ satisfies the assumption of Theorem~\ref{thm:general1} (as, using l'Hospital's rule, we have $\lim_{x\to 0^+}\frac{g(x)}{x}=1$, so $\gamma=1$ and $M=1$ work), but it does not satisfy the assumption of Theorem~\ref{thm:general2} (as $\lim_{x\to\infty}\frac{g(x)}{g(x/2)}=\infty$, so if $\varepsilon=1/2$, then for every $M>0$ one can find $x>0$ such that $\frac{g(x)}{g(x/2)}>M$).

 \item 
 On the other hand, Theorem~\ref{thm:general2} works for any sequences $(c_n)$ whereas Theorem~\ref{thm:general1} works only for sequences $(c_n)$ which are divergent to infinity.

\item If a  function $g(x)$ satisfies the assumptions of Theorem~\ref{thm:general1}, it also satisfies $\lim_{x\to 0^+}g(x)=0$. On the other hand, if $g(x) = e^{-1/x}$, then $\lim_{x\to 0^+} g(x) = 0$, but $g(x)$ does not satisfy the assumption of Theorem~\ref{thm:general1}. 
Moreover, the equivalence from Theorem~\ref{thm:general1} does not hold for the function $g(x)=e^{-1/x}$ as it is witnessed by  sequences 
$a_n=1/n^2$, $b_n=1$ and $c_n=\ln n$.
\end{enumerate}
\end{remark}

\begin{corollary}
If $g$, $(b_n)$ and $(c_n)$ are like in Theorem~\ref{thm:general1} or Theorem~\ref{thm:general2}, then
for every sequence $(a_n)$ with  $a_n\in\ran(g)$ we have
$$\sum_{n=1}^{\infty} b_n a_n<\infty \implies   \I_{\left(b_n g(1/c_n)\right)}^*-\lim c_n g^{-1}(a_n)=0  $$
and
$$ \sum_{n=1}^{\infty} b_n a_n<\infty \implies   \I_{\left(b_n g(1/c_n)\right)}-\lim c_n g^{-1}(a_n)=0.$$    
\end{corollary}

\begin{proof}
Apply Theorem~\ref{thm:general1} or Theorem~\ref{thm:general2} with the ideal $\I =    \I_{\left(b_n  g(1/c_n)\right)}$.
\end{proof}

The equivalence  ``$(\ref{cor:Salat-Toma:I-conv}) \iff (\ref{cor:Salat-Toma:extension})$'' in the following result is just Theorem~\ref{thm:salat-toma} and it was proved in \cite{MR2031274}. Here, we strengthen  this theorem essentially, because $\I^*$-convergence is stronger than $\I$-convergence.

\begin{corollary}
\label{cor:Salat-Toma}
Let $\I$ be an ideal on $\N$. The following conditions are equivalent. 
\begin{enumerate}
    \item For every  sequence $(a_n)$ of non-negative numbers we have\label{cor:Salat-Toma:I-conv} 
$$\sum_{n=1}^\infty a_n<\infty \implies \I-\lim n a_n=0.$$
    \item For every  sequence $(a_n)$ of non-negative numbers we have\label{cor:Salat-Toma:Istar-conv} 
$$\sum_{n=1}^\infty a_n<\infty \implies \I^*-\lim n a_n=0.$$
\item The ideal $\I$ extends the summable ideal:\label{cor:Salat-Toma:extension} 
$\I_{1/n}\subseteq  \I.$
\end{enumerate}
\end{corollary}

\begin{proof}
Apply Theorem~\ref{thm:general1} with  
$g(x)=x$, $b_n=1$ and $c_n=n$ and Remark~\ref{non-negative}(\ref{non-negative:non-negative-instead-positive}). 
\end{proof}

\begin{corollary}[{Mi\v{s}\'{\i}k-T\'{o}th, ~\cite{MR4327685}}]
Let $\I$ be an ideal on $\N$.
Let $p, q$ be fixed positive numbers and $\alpha,\beta$ be fixed nonnegative numbers. 
Then the following conditions are equivalent.
\begin{enumerate}
\item For every sequence $(d_n)$ of positive numbers we have
$$\sum_{n=1}^{\infty} n^\alpha d_n^p<\infty\implies  \I-\lim n^\beta d_n^q=0.$$

\item 
The ideal $\I$ extends the summable ideal generated by the sequence $(n^{\alpha - \beta p/q})$:
$$\I_{\left(n^{\alpha - \beta p/q}\right)}\subseteq\I.$$
\end{enumerate}
\end{corollary}

\begin{proof}
We can apply Theorem~\ref{thm:general2} with  
$g(x)=x^{p/q}$,
$a_n=d_n^p$, $b_n=n^\alpha$ and $c_n=n^\beta$. 
\end{proof}

\begin{corollary}[{Bartoszewicz-G\l{}\c{a}b-Widz~\cite{MR4317877}}]
\label{cor:Bartoszewicz-Glab-Widz-series}
Let $\I$ be an ideal on $\N$.
Let $(b_n)$, $(c_n)$ be sequences of positive numbers and let $p$, $q$ be fixed positive numbers. 
Then the following conditions are equivalent.
\begin{enumerate}
\item For every sequence $(d_n)$ of positive numbers we have
$$\sum_{n=1}^{\infty} b_n d_n^p<\infty\implies  \I-\lim c_n d_n^q=0.$$

\item The ideal $\I$ extends the summable ideal generated by the sequence $(b_n c_n^{-p/q})$:
$$\I_{\left(b_n c_n^{-p/q}\right)}\subseteq\I.$$
\end{enumerate}
\end{corollary}

\begin{proof}
We can apply Theorem~\ref{thm:general2} with  
$g(x)=x^{p/q}$,
$a_n=d_n^p$, $b_n=b_n$ and $c_n=c_n$. 
\end{proof}

\begin{remark}
One can show that for instance functions $g(x)=e^x-1$, $g(x)=\ln(x+1)$, $g(x)=\arctan x$ and even $g(x)=\Phi(x)-1/2$ (where $\Phi(x)$ is the cumulative distribution function of the standard normal distribution) satisfy the assumptions of Theorem~\ref{thm:general1}. 
On the other hand, all these functions are not the power functions $x^{p/q}$ considered in Corollary~\ref{cor:Bartoszewicz-Glab-Widz-series}.
\end{remark}


\subsection{$\I^+$ test}
\label{subsec:general3}

\begin{theorem}\label{thm:general3}
Let $\I$ be an ideal on $\N$.
Let $g$, $(b_n)$ and $(c_n)$ be like in Theorem~\ref{thm:general1} or Theorem~\ref{thm:general2}. 
Then the following conditions are equivalent.
\begin{enumerate}
\item For every sequence $(a_n)$ with  $a_n\in\ran(g)$ we have\label{thm:general3:I-plus-convergence}
$$\sum_{n=1}^{\infty} b_n a_n<\infty \implies   \I^+-\lim c_n g^{-1}(a_n)=0.$$

\item The filter dual to $\I$ is disjoint from the summable ideal generated by the sequence $(b_n g(1/c_n))$:
$$ \I_{\left(b_n g(1/c_n)\right)}\cap \I^*=\emptyset. $$\label{thm:general3:ideals}
\end{enumerate}
\end{theorem}

\begin{proof}
$(\ref{thm:general3:I-plus-convergence})\implies (\ref{thm:general3:ideals})$
Suppose that there exists $A\in \I_{\left(b_n g(1/c_n)\right)}\cap\I^*$. For any $n\in\N$, take as $d_n$ some element of $\ran(g)$ with $d_n\leq \frac{1}{n^2 b_n}$. We can find such an element by the assumption $\lim_{x\to 0^+} g(x)=0$.
 We define
$$a_n=\left\{\begin{array}{ll}
g(1/c_n) & \textrm{ for } n\in A,\\
d_n & \textrm{ for } n\not\in A. \end{array}\right. $$

We can see that $\sum_{n=1}^{\infty}  b_n a_n<\infty$ because
$$\sum_{n=1}^{\infty}b_n a_n\leq \sum_{n\in A} b_n g(1/c_n) + \sum_{n\not\in A} \frac{b_n}{b_n n^2}.$$
Since $A\in \I_{\left(b_n g(1/c_n)\right)}$, the first sum is finite and the second sum is finite because the series $\sum \frac{1}{n^2}$ is convergent. 

On the other hand, for all $n\in A$ we have 
$$ c_n g^{-1}(a_n)=c_n g^{-1}(g(1/c_n))=c_n (1/c_n)=1.$$
Since $A\in\I^*$, it follows that for any $B\in\I^+$ there are infinitely many $n\in B\cap A$ with $c_n g^{-1}(a_n)=1$, thus we cannot have $\I^+-\lim c_n g^{-1}(a_n)=0$.

\smallskip

$(\ref{thm:general3:ideals})\implies (\ref{thm:general3:I-plus-convergence})$
Suppose that there exists a positive sequence $(a_n)$ with $\sum_{n=1}^\infty b_n a_n<\infty $ such that for any $B\in\I^+$ we have $\limsup_{n\in B}  c_n g^{-1}(a_n)>0$.
Consider the sets $A_k=\{n\in\N:  c_n g^{-1}(a_n)\geq 1/k\}$. We can notice that for each $k\in\N$ we have $A_k\in \I_{\left(b_n g(1/c_n)\right)}$. Indeed, let us assume that it is not the case for some $k\in\N$. Then
$$\infty>\sum_{n\in\N} b_n a_n \geq \sum_{n\in A_k}b_n a_n \geq \sum_{n\in A_k}b_n g\left(\frac{1}{k c_n} \right),  $$ 
which is infinite since $A_k\notin \I_{\left(b_n g(1/c_n)\right)}$ and 
$$\limsup_{n\to\infty}\frac{g(1/c_n)}{g(\varepsilon/c_n)}\in(0,\infty)  $$  
for any $\varepsilon>0$ by the reasonings presented in the proofs of Theorems~\ref{thm:general1} and \ref{thm:general2}, thus bringing us to a contradiction.

Now, since $\I_{\left(b_n g(1/c_n)\right)}$ is a P-ideal, there exists $B\in\I_{\left(b_n g(1/c_n)\right)}^*$ such that $B\cap A_k$ is finite for all $k\in\N$. We can see that $\lim_{n\in B}c_n g^{-1}(a_n)=0$. By our assumption we get $B\not\in\I^+$, hence $B\in\I$. If we now take $C=\N\setminus B$, we obtain $C\in \I_{\left(b_n g(1/c_n)\right)}\cap\I^*$. 
\end{proof}

\begin{remark}
\label{non-negative2}
Notice that Theorem~\ref{thm:general3} would still be true if we add $0$ to the domain and codomain of $g$ and require that $g(0)=0$.
\end{remark}

\begin{corollary}\label{cor:square-plus}
Let $\I$ be an ideal on $\N$.
Then the following conditions are equivalent.
\begin{enumerate}
\item For every sequence $(a_n)$ of non-negative numbers  we have
$$\sum_{n=1}^{\infty} a_n<\infty \implies   \I^+-\lim n a_n=0.$$

\item 
The filter dual to $\I$ is disjoint from the summable ideal:
$$ \I_{1/n}\cap \I^*=\emptyset. $$
\end{enumerate}
\end{corollary}

\begin{proof}
Apply Theorem~\ref{thm:general3} with  
$g(x)=x$,
 $b_n=1$ and $c_n=n$ and Remark~\ref{non-negative2}. 
\end{proof}


\section{The ideal test for the divergence of an infinite series with monotone terms}
\label{sec:general4-monotone}

\begin{definition}
For an infinite set $X=\{x_1<x_2<\ldots\}\subseteq\N$, we define $f_X:\N\to\N$ by  
$$f_X(i)=\frac{1}{x_n} \iff i\in(x_{n-1}, x_n] \text{ for some } n\in\N$$ 
(we take $x_0=0$), and 
$$\I_X=\left\{A\subseteq\N: \sum_{n\in A} f_X(n)<\infty \right\}.$$ 
\end{definition}

The following easy proposition summarizes few basic properties of ideals of the form $\I_X$.

\begin{proposition}\label{prop:IXproperties} \ 
\begin{enumerate}
\item $\I_{\N}=\I_{1/n}$.
    \item 
For any infinite $X$, $\I_X$ is equal to the summable ideal generated by the sequence $f_X$:   $\I_X=\I_{(f_X)}$.
\item If $X\subseteq Y$ then $\I_X\supseteq\I_Y$. 
\item $\I_{1/n}\subseteq \I_X$ for every infinite $X$.
\item For any infinite $X$, if $A\in\I_X^*$ then $A$ has upper asymptotic density $1$:
$$A\in\I^*_X\Longrightarrow\limsup_{n\to\infty}\frac{|A\cap\{1,\ldots,n\}|}{n}=1. $$\label{prop:IXproperties:upper-density}
\end{enumerate}
\end{proposition}
\begin{proof}
The first four items are easy observations. We will prove the last item by showing that if $A$ has positive lower asymptotic density then $A\not\in\I_X$, i.e.
$$\liminf_{n\to\infty}\frac{|A\cap\{1,\ldots,n\}|}{n}>0\Longrightarrow A\not\in\I_X.$$

Take $A\subseteq\N$ with positive lower asymptotic density. Pick $\alpha>0$ such that the lower asymptotic density of $A$ is greater than $2\alpha$. 
First, observe that there exist $k,N\in\N$ such that for  all $n\geq N$ we have $|A\cap[2^{nk},2^{nk+k})|/2^{nk+k}>\alpha$. Indeed, otherwise for infinitely many $n\in\N$ we have 
$$2\alpha<\frac{|A\cap[1,2^{nk+k}]|}{2^{nk+k}}\leq \frac{2^{nk}}{2^{nk+k}}+ \frac{|A\cap[2^{nk},2^{n+k})|}{2^{nk+k}}\leq 2^{-k}+\alpha, $$
which is a contradiction for any $k\in\N$ with $2^{-k}\leq\alpha$.

Now, for any $n\in\N$ denote by $I_n$ the interval $[2^{nk},2^{nk+k})$. Let $Y=\{y_1<y_2<\ldots\}$ be such~a subset of $X$ that $|Y\cap I_n|\leq 1$ for all $n\in\N$ and $y_1>\max I_N$. Since $\I_X\subseteq\I_Y$ by (3), we will finish the proof by showing that $A\not\in \I_Y$.

Take any $y_n$. Then $y_n\in I_m$ for some $m>N$, thus
$$\sum_{i\in A\cap I_{m-1}}f_Y(i)\geq \frac{\alpha 2^{mk}}{y_n}\geq \frac{\alpha  2^{mk}}{2^{mk+k}}=\frac{\alpha}{2^k}. $$
Since that calculation holds for any $y_n$ and $Y$ is infinite, we obtain 
$$\sum_{i\in A}f_Y(i)\geq \sum_{n=1}^\infty \frac{\alpha}{2^k}=\infty, $$
hence $A\not\in\I_Y.$
\end{proof}

\begin{theorem}\label{thm:star-ideal}
Let $\I$ be an ideal on $\N$. 
Then the following conditions are equivalent.
\begin{enumerate}
\item For every $\I^*$-nonincreasing sequence  $(a_n)$ of positive reals we have\label{thm:star-ideal:I-star-convergence}
$$\sum_{n=1}^{\infty} a_n<\infty \implies   \I^*-\lim n a_n=0.$$

\item For every $\I^*$-nonincreasing sequence  $(a_n)$ of positive reals we have\label{thm:star-ideal:I-convergence}
$$\sum_{n=1}^{\infty} a_n<\infty \implies   \I-\lim n a_n=0.$$

\item  
The filter dual to $\I$ is disjoint from each ideal $\I_X$ with  $X\in \I^+$:
$$\forall X\in\I^+ \, \left( \I_{X}\cap \I^*=\emptyset\right).$$\label{thm:star-ideal:ideals}
\end{enumerate}
\end{theorem}

\begin{proof}
$(\ref{thm:star-ideal:I-star-convergence})\implies (\ref{thm:star-ideal:I-convergence})$ 
 It follows from the fact that $\I^*$-convergence implies $\I$-convergence.

\smallskip

$(\ref{thm:star-ideal:I-convergence})\implies (\ref{thm:star-ideal:ideals})$ 
Suppose there exist $X\in\I^+$ and $A\subseteq\N$ such that $A\in\I_{X}\cap \I^*\not=\emptyset$.
Define
$$a_n=
\begin{cases}
f_X(n) & \text{ for } n\in A,\\
1/n^2 & \text{ for } n\not\in A. 
\end{cases} $$

Since $f_X$ is nonincreasing, the sequence $a_n$ is nonincreasing on the set $A\in\I^*$. Moreover, $$\sum_{n\in\N}a_n=\sum_{n\in A} f_X(n)+\sum_{n\not\in A} \frac{1}{n^2} <\infty,$$
because $A\in\I_X$.

On the other hand, for every $n=x_k\in X\cap A$ we have 
$$n a_n=x_k f_X(x_k)=x_k\cdot \frac{1}{x_k}=1.$$
Since $X\cap A\in\I^+$,  the sequence $(n a_n)$ cannot be $\I$-convergent to zero.

\smallskip

$(\ref{thm:star-ideal:ideals})\implies (\ref{thm:star-ideal:I-star-convergence})$ 
Suppose there exists a sequence $(a_n)$ with $\sum_{n\in\N}a_n<\infty$ that is nonincreasing on some set $A\in\I^*$ and  the sequence $(n a_n)$ is not $\I^*$-convergent to zero. 

We have two cases.
\begin{enumerate}
    \item[(a)] There is an $\varepsilon>0$ such that $\{n\in A: n a_n>\varepsilon\}\in\I^+$. 

\item [(b)]
For every $\varepsilon>0$ we have $\{n\in A: n a_n>\varepsilon\}\in\I$
\end{enumerate}

In case (a), let $\varepsilon>0$ be such that $X = \{n\in A: n  a_n>\varepsilon\}\in \I^+$, and   enumerate the elements of $X$ increasingly by $x_1,x_2,\ldots$. 

Since the sequence $(a_n)$ is nonincreasing on $A$ and $X\subseteq A$, we can notice that $a_n\geq a_{x_k}$ for $n\in (x_{k-1},x_k]\cap A$, $k\in\N$.
Therefore,
\begin{equation*}
    \begin{split}
        \sum_{n\in A} a_n
        & \geq 
        \sum_{k\in\N} \sum_{n\in (x_{k-1},x_k]\cap A}a_{x_k}
        >
        \sum_{k\in\N} \sum_{n\in (x_{k-1},x_k]\cap A}\frac{\varepsilon}{x_k}
        \\&=
        \varepsilon\sum_{k\in\N} \sum_{n\in (x_{k-1},x_k]\cap A}\frac{1}{x_k}
        =
        \varepsilon \sum_{n\in A} f_X(n). 
    \end{split}
\end{equation*}
From the assumption that  $\sum_{n\in A} a_n<\infty$, it follows that $A\in\I_X$. Thus, $A\in\I_X\cap\I^*$, which makes $\I_X\cap\I^*$ nonempty.

In case (b), since the sequence $(n a_n)$ is not $\I^*$-convergent to $0$ and $A\notin\I$, we can find a strictly decreasing sequence $(\varepsilon_k)$ tending to $0$ such that $X_k=\{n\in A: n a_n\in [\varepsilon_k,\varepsilon_{k-1})\}\in \I\setminus\fin$ for every $k\in\N$ (we put $\varepsilon_0=\infty$). Observe also that for every $B\in\I^*$ there is some $k\in\N$ with $B\cap X_k\not\in\fin$. Enumerate elements of each $X_k$ increasingly by $x_1^{(k)},x_2^{(k)},\ldots$ and add $x_0^{(k)}=0$.

We will prove that $A\in\I_{X_k}$ for every $k\in\N$. Take any $k\in\N$ and notice that for any $n\in X_k$ we have $a_n\geq\varepsilon_k/n$, thus, using the fact that $(a_n)$ is nonincreasing on $A$, we have
$$\sum_{n\in A}f_{X_k}(n)=\sum_{i\in \N}\frac{|A\cap (x_{i-1}^{(k)},x_i^{(k)}] |}{x_i^{(k)}}\leq \frac{1}{\varepsilon_k}\sum_{n\in A} a_n<\infty.  $$

Because $\sum_{n\in A}f_{X_k}(n)<\infty$ for each $k\in\N$, we can see that we may always find such $t_k\in\N$ that 
$$\sum_{i\geq t_k}\frac{|A\cap (x_{i-1}^{(k)},x_i^{(k)}] }{x_i^{(k)}}<\frac{1}{2^k}.$$
Moreover, by increasing $t_k$ if necessary, we can assume that for  each $k>1$ there exist some $j\geq t_1$ such that $x_j^{(1)}\in (x_{t_k-1}^{(k)},x_{t_k}^{(k)}) $.

Next, define $X=\bigcup_{k\in\N}X_k\setminus\{x_1^{(k)},\ldots,x_{t_k-1}^{(k)}\}$. Note that $X\in\I^+$ as otherwise \mbox{$\N\setminus X$} would be a set in $\I^*$ that has finite intersections with every $X_k$. Enumerate increasingly elements of $X$ by $x_1,x_2,\ldots$ and add $x_0=0$. Observe that $x_1=x_{t_1}^{(1)}$.

We will finish the proof by showing that $A\in \I_X$. In order to prove that,  observe that every $x_j$ is equal to $x_i^{(k)}$ for some $k\in\N$ and $i\geq t_k$. Moreover, for every $x_j=x_i^{(k)}$ other than $x_1=x_{t_1}^{(1)}$ we can notice that $(x_{j-1},x_j]\subseteq (x_{i-1}^{(k)}, x_i^{(k)}]$ because either $i>t_k$ and then $x_{i-1}^{(k)}\in X$, thus $x_{i-1}^{(k)}\leq x_{j-1}$, or $i=t_k$ and then  $X_1\cap X\cap (x_{t_k-1}^{(k)},x_{t_k}^{(k)})\not=\emptyset $, thus $x_{i-1}^{(k)}< x_{j-1}$. Therefore,
\begin{equation*}
    \begin{split}
\sum_{n\in A\setminus\{1,\ldots,x_1\}}f_{X}(n)
&=
\sum_{j\geq 2}\frac{|A\cap(x_{j-1},x_j]|}{x_j}
\\&\leq  
\sum_{k\in\N}\sum_{i\geq t_k}\frac{|A\cap (x_{i-1}^{(k)},x_i^{(k)}] |}{x_i^{(k)}}
< 
\sum_{k\in\N}\frac{1}{2^k}=1<\infty. 
\end{split}
\end{equation*}
It clearly follows that $\sum_{n\in A}f_{X}(n)<\infty$, thus $A\in\I_X\cap\I^*$.
\end{proof}

\begin{theorem}\label{thm:star-plus}
Let $\I$ be an ideal on $\N$. 
Then the following conditions are equivalent.
\begin{enumerate}
\item For every $\I^*$-nonincreasing sequence  $(a_n)$ of positive reals we have\label{thm:star-plus:I-plus-convergence}
$$\sum_{n=1}^{\infty} a_n<\infty \implies   \I^+-\lim n a_n=0.$$

\item 
The filter dual to $\I$ is disjoint from the summable ideal:
$$ \I_{1/n}\cap \I^*=\emptyset. $$\label{thm:star-plus:ideals}
\end{enumerate}
\end{theorem}
\begin{proof}
$(\ref{thm:star-plus:ideals})\implies (\ref{thm:star-plus:I-plus-convergence})$
It follows from Corollary~\ref{cor:square-plus}.

$(\ref{thm:star-plus:I-plus-convergence})\implies (\ref{thm:star-plus:ideals})$
Suppose that there exists $A\in \I_{1/n}\cap\I^*$.
We define
$a_n=1/n$ for $n\in A$
and $a_n=1/n^2$ for $n\not\in A$.
Then $(a_n)$ is $\I^*$-nonincreasing and we can see that 
$$\sum_{n=1}^{\infty}a_n=\sum_{n\in A} \frac{1}{n} + \sum_{n\not\in A} \frac{1}{n^2}<\infty,$$
because  $A\in \I_{1/n}$. 

On the other hand, for all $n\in A$ we have $n a_n=1$. Since $A\in\I^*$, it follows that for any $B\in\I^+$ there are infinitely many $n\in B\cap A$ with $n a_n=1$, thus we cannot have $\I^+-\lim n a_n=0$.
\end{proof}

\begin{corollary}[Faisant-Grekos-Mi{\v{s}}{\'i}k~\cite{MR3576795}]
\label{thm:fgm}
If  $\I$ is an ideal on $\N$ such that every $A\in\I$ has the upper asymptotic density less than 1, then  
$$\sum_{n=1}^{\infty} a_n<\infty \implies   \I^*-\lim n a_n=0$$
for every $\I^*$-nonincreasing sequence  $(a_n)$ of positive reals.
\end{corollary}

\begin{proof}
Since every $A\in\I$ has the upper asymptotic density less than 1, by Proposition~\ref{prop:IXproperties}(\ref{prop:IXproperties:upper-density}) it follows that $\I\cap \I_X^*=\emptyset$ for every infinite $X\subseteq\N$. Hence, in particular, $\I_X\cap \I^*=\emptyset$ for every $X\not\in\I$. Therefore, by Theorem~\ref{thm:star-ideal} we obtain the desired result.
\end{proof}

\begin{proposition}
\label{prop:necessary-and-sufficient-conditions}
Let $\I$ be an ideal on $\N$.
In the following list of conditions, each implies the next
and none of the implications reverse.
\begin{enumerate}
    \item 
$\I_{1/n}\cap\I^+=\emptyset$.
\item $\I_{X}\cap \I^*=\emptyset$ for every $X\in\I^+$.

    \item 
$\I_{1/n}\cap\I^*=\emptyset$.

\end{enumerate}
\end{proposition}
\begin{proof}
$(1)\implies(2)$
Suppose that $\I_{1/n}\cap\I^+=\emptyset $.
Then using Theorems~\ref{thm:salat-toma} and \ref{thm:star-ideal} we obtain 
$ \I_{X}\cap \I^*=\emptyset$ for every $X\in\I^+$.

\smallskip

$(2)\implies(3)$
Suppose that $ \I_{X}\cap \I^*=\emptyset$ for every $X\in\I^+$.
Taking $X=\N$ we get $\I_X=\I_{1/n}$, so 
$\I_{1/n}\cap\I^*=\emptyset$.

\smallskip

$ (2)\;\not\!\!\!\!\implies(1)$
Consider a summable ideal $\I=\I_{1/\ln (n+1)}$. 
Then $\I\subseteq\I_{1/n}$, and since 
$\{n^2:n\in\N\}\in\I_{1/n}\setminus\I$,
we have $\I_{1/n}\cap\I^+\not=\emptyset$. 
On the other hand, since $\I\subseteq \I_{1/n}\subseteq\I_X$ for every $X$, we have $\I_{X}\cap \I^*=\emptyset$ for every $X$.

\smallskip

$ (3)\;\not\!\!\!\!\implies(2)$
Let $k_n=\lfloor n\ln (n+1) \rfloor$ for $n\in\N$ and define $K=\{k_n:n\in\N\}$. Take $\I=\{A\subseteq\N: A\cap K\in\fin\}$. 
Notice that if $A\in\I^*$ then $K \setminus A$ is finite  and
$$\sum_{n\in K} \frac{1}{n}=\sum_{n\in\N}\frac{1}{k_n}\geq \sum_{n\in\N}\frac{1}{n \ln(n+1)}=\infty, $$
thus $K\not\in\I_{1/n}$, hence $\I_{1/n}\cap\I^*=\emptyset$.

Now, we pick the sequence $i_1<i_2<\ldots$ in such a way that $i_n/k_{i_n}<2^{-n}$. We can do it because $\lim_{n\to\infty}n/k_n=0$.
Consider the set $A=\{k_{i_n}:n\in\N\}$. Then $A\in\I^+$ as $A\subseteq K$ and $A$ is infinite.
Moreover,
$$\sum_{k\in K}f_A(k)\leq\sum_{n\in\N}(i_n-i_{n-1})\frac{1}{k_{i_n}} \leq \sum_{n\in\N}\frac{i_n}{k_{i_n}}< \sum_{n\in\N}\frac{1}{2^n} <\infty.$$

Therefore, there is $A\in\I^+$ such that $K\in\I_A\cap\I^*$, thus $\I_A\cap\I^*\not=\emptyset$.
\end{proof}

%
%

\begin{corollary}
Let $\I$ be an ideal on $\N$.
In the following list of conditions, each implies the next
and none of the implications revers.
\begin{enumerate}
\item  
The coideal of $\I$ is disjoint from the summable ideal:
$$\I_{1/n}\cap\I^+=\emptyset.$$ 
\item For every $\I^*$-nonincreasing sequence  $(a_n)$ of positive reals we have
$$\sum_{n=1}^{\infty} a_n<\infty \implies   \I^*-\lim n a_n=0.$$

\item  
The filter dual to $\I$ is disjoint from the summable ideal:
$$\I_{1/n}\cap\I^*=\emptyset.$$
\end{enumerate}
\end{corollary}

\begin{proof}
Use Theorem~\ref{thm:star-ideal} along with Proposition~\ref{prop:necessary-and-sufficient-conditions}.
\end{proof}


\section{Algebraic structures in families of sequences related to Olivier's theorem}
\label{sec:algebraic-structures}

The following proposition gives a necessary and sufficient conditions under which the families $AOS(\I)$, $AOS(\I^+)$ and $AOS(\I^+)$ are nonempty.

\begin{proposition}
\label{prop:AOSs-nonempty}
Let $\I$ be an ideal on $\N$.
\begin{enumerate}
    \item The following conditions are equivalent.
    \begin{enumerate}
        \item $AOS(\I)\neq\emptyset$.
        \item $AOS(\I^*)\neq\emptyset$.
        \item $ \I^+\cap \I_{1/n}\neq\emptyset$.
    \end{enumerate}
    \item The following conditions are equivalent.
    \begin{enumerate}
        \item $AOS(\I^+)\neq\emptyset$.
        \item $ \I^*\cap \I_{1/n}\neq\emptyset$.
    \end{enumerate}
    \end{enumerate}
\end{proposition}

\begin{proof}
(1) It follows from Corollary~\ref{cor:Salat-Toma}.
(2) It follows from Corollary~\ref{cor:square-plus}.
\end{proof}

Since $\I^*$-convergence implies both $\I$-convergence and $\I^+$-convergence,  
$AOS(\I)\subseteq AOS(\I^*)$ and $AOS(\I^+)\subseteq AOS(\I^*)$. 
Below, we show that, in general,  these inclusions do not reverse, and there is no inclusions between $AOS(\I)$ and $AOS(\I^+)$.

\begin{proposition}
\label{prop:relations-between-AOSs}
Let $\I$ be an ideal on $\N$.
\begin{enumerate}
    \item 
If $\I^+\cap \I_{1/n}\neq\emptyset$ and  $\I^*\cap \I_{1/n}=\emptyset$, then $AOS(\I)\not\subseteq AOS(\I^+)$ and $AOS(\I^*)\not\subseteq AOS(\I^+)$.

\item Every ideal $\I$ which is strictly contained in $\I_{1/n}$ (e.g.~$\I=\fin$) satisfies assumptions of item (1).

\item If $\I$ is not a weak P-ideal and $\I^*\cap \I_{1/n}\neq \emptyset$, then
$AOS(\I^+)\not\subseteq AOS(\I)$.

\item If $\I$ is not a P-ideal and $\I^*\cap \I_{1/n}\neq \emptyset$, then
$AOS(\I^*)\not\subseteq AOS(\I)$.

\item There exists an  ideal $\I$ which satisfies the assumptions of items (3) and (4).

\end{enumerate}
\end{proposition}

\begin{proof}
(1)
By Proposition~\ref{prop:AOSs-nonempty}, we have  $AOS(\I)\neq\emptyset$, $AOS(\I^*)\neq\emptyset$ and $AOS(\I^+)=\emptyset$. 

(2) It is obvious.

(3)
Take $A\in\I^*\cap\I_{1/n}$. Let sets $A_1,A_2\ldots\in\I$ be pairwise disjoint such that for any $B\not\in\I$ there is $n\in\N$ with $B\cap A_n\not\in\fin$. We may assume that  $\bigcup_{n=1}^\infty A_n=A$. 
Indeed, if that is not the case then enumerate $A\setminus \bigcup_{n=1}^\infty A_n$ by $(x_i)$ (for either $i\in\N$ or $i\leq N$ depending on whether this difference is infinite or not) and for each $n\in\N$ put $A_n'=(A_n\cap A)\cup \{x_n\}$ (or $A'_n=A_n \cap A$ in case $x_n$ is not defined). Then clearly $A'_n\in\I$ for every $n\in\N$ and $\bigcup_{n=1}^\infty A'_n=A$.  
We define 

$$a_n=\begin{cases} 
\frac{1}{n i} & \textrm{ for } n\in A_i,\\
\frac{1}{n^2} & \textrm{ for other } n. 
\end{cases}$$

Obviously  
$$\sum_{n=1}^\infty a_n\leq \sum_{n=1}^\infty \frac{1}{n^2} +\sum_{n\in A} \frac{1}{n}<\infty.  $$

Moreover, we can notice that for every $k\in\N$ we have
$$\left\{n\in\N:n a_n\geq \frac{1}{k}\right\}\subseteq (\N\setminus A)\cup \bigcup_{i=1}^k A_i\in\I,$$
thus $\I-\lim n a_n=0,$ hence $(a_n)\not\in AOS(\I)$.

On the other hand for every $B\not\in\I$ there is $k\in\N$ with $B\cap A_k\not\in\fin$, thus there are infinitely many $n\in B$ such that $n a_n=1/k$, hence $(n a_n)_{n\in B}$ cannot be convergent to $0$. It follows that $(a_n)\in AOS(\I^+)$.

(4)
The obvious modification of the proof of item (3)  gives the proof of item (4).

(5) 
Take an infinite set  $A\in \I_{1/n}$.
Let $\{A_n : n\in \N\}$ be an infinite partition of $A$ into infinite sets.
We define an ideal $\I$ by
$$B\in \I \iff B\cap A_n\in \fin \text{\ \ for all but finitely many $n\in \N$}.$$

Then $\I$ is not a weak P-ideal (so also not a P-ideal) and $A\in \I^*\cap \I_{1/n}$.
\end{proof}


\subsection{Lineability}
\label{subsec:lineability}

\begin{theorem}
\label{the:lineability}
Let $\I$ be an ideal on $\N$.
\begin{enumerate}
    \item
    The following conditions are equivalent.
\begin{enumerate}
    \item $AOS(\I)\neq\emptyset$.
    \item $AOS(\I^*)\neq\emptyset$.
    \item $AOS(\I)$ is $\mathfrak{c}$-lineable.
    \item $AOS(\I^*)$ is $\mathfrak{c}$-lineable.
    \item $\I_{1/n}\cap \I^+\neq\emptyset$.
\end{enumerate}
\item 
The following conditions are equivalent.
\begin{enumerate}
    \item $AOS(\I^+)\neq\emptyset$.
    \item $AOS(\I^+)$ is $\mathfrak{c}$-lineable.
    \item $\I_{1/n}\cap \I^*\neq\emptyset$.
\end{enumerate}
\end{enumerate}
\end{theorem}

\begin{proof}
The equivalence of (1a), (1b) and (1e) follows from Proposition~\ref{prop:AOSs-nonempty}. 
The implications (1c)$\implies$(1a) and (1d)$\implies$(1b) are obvious. 
By $AOS(\I)\subseteq AOS(\I^*)$ we get (1c)$\implies$(1d). Thus, it suffices to show (1e)$\implies$(1c).

The equivalence of  (2a) and (2c) follows from Proposition~\ref{prop:AOSs-nonempty}. The implication (2b)$\implies$(2a) is obvious. Thus, it suffices to show (2c)$\implies$(2b).

Below we prove that (1e)$\implies$(1c) ((2c)$\implies$(2b), resp.).

Assume that there is some $A\in\I_{1/n}\cap \I^+$ ($A\in\I_{1/n}\cap \I^*$, resp.). 

Since $\sum_{n\in A}1/n<\infty$, we can find  an increasing sequence $(j_k)$ of integers such that 
$$\sum_{n>j_k,n\in A}\frac{1}{n}<\frac {1}{k2^k}.$$ 
We put $A_0=(\N\setminus A)\cup[1,j_1]\cap \N$ and $A_k=A\cap(j_{k},j_{k+1}]$ for each $k\geq 1$. 
Observe that $(A_k)$ is a partition of $\N$ and $A_k\in\I$ for $k\geq 1$, while $A_0\notin\I^*$ ($A_k\in\I$ for all $k$, resp.). Moreover, 
$$\sum_{k=0}^\infty \left(k\sum_{n\in A_k}\frac{1}{n}\right)<\infty.$$

 For each $\alpha\in (0,1)$ let $x^{(\alpha)}$ be a sequence given by $$x^{(\alpha)}(n)=\frac{k^\alpha}{n} \iff n\in A_k.$$

Note that $x^{(\alpha)}\in \ell_1$ for each $\alpha\in(0,1)$ as
$$\sum_{n=1}^\infty x^{(\alpha)}(n)=\sum_{k=0}^\infty \left(k^\alpha\sum_{n\in A_k}\frac{1}{n}\right)\leq \sum_{k=0}^\infty \left(k\sum_{n\in A_k}\frac{1}{n}\right)<\infty.$$

In order to show $\mathfrak{c}$-lineability of $AOS(\I)$ ($AOS(\I^+), resp.)$, consider a linear combination 
$$y=c_1 x^{(\alpha_1)}+\ldots+c_m x^{(\alpha_m)},$$ 
where $c_i\in\R\setminus\{0\}$ for $i\leq m$ and $\alpha_1>\ldots>\alpha_m$. We need to show that $y\in AOS(\I)$ ($y\in AOS(\I^+)$, resp.).

Obviously, $y\in\ell_1$ as a linear combination of $\ell_1$-sequences. 

Observe that for each $n\in A_k$ with $k\geq 1$ we have
\begin{equation*}
    \begin{split}
|n y(n)|
&=
\left|n  \sum_{i=1}^m c_i\cdot \frac{k^{\alpha_i}}{n}\right|
=
\left|\sum_{i=1}^m c_i k^{\alpha_i}\right|
=
|c_1k^{\alpha_1}|\cdot \left|1+ \sum_{i=2}^m \frac{c_i}{c_1}\cdot k^{\alpha_i-\alpha_1}\right|
\\&\geq
|c_1k^{\alpha_1}|\cdot \left(1- \left|\sum_{i=2}^m \frac{c_i}{c_1}\cdot k^{\alpha_i-\alpha_1}\right|\right)
\xrightarrow{k\to\infty}\infty,
    \end{split}
\end{equation*}
as 
$$\lim_{k\to\infty} k^{\alpha_1}=\infty \text{ \ and \ } 
\lim_{k\to\infty} k^{\alpha_i-\alpha_1}=\lim_k\frac{1}{k^{\alpha_1-\alpha_i}}=0 \text{\ for all\ } i=2,\ldots,m.$$

To show that $y\in AOS(\I)$, find $k_0$ such that $|n y(n)|\geq 1$ for all $n\in\bigcup_{k\geq k_0}A_k$ and note that $\bigcup_{k\geq k_0}A_k\notin\I$ as $A_0\notin\I^*$ and $A_k\in\I$ for $k\geq 1$. Hence, $y\in AOS(\I)$.

To show that $y\in AOS(\I^+)$, fix any $B\in\I^+$. Since $A_k\in\I$ for all $k\in\N\cup\{0\}$, the set $\{k\in\N:\ B\cap A_k\neq\emptyset\}$ has to be infinite. Thus, $(|n y(n)|)_{n\in B}$ contains a subsequence  $(|n y(n)|)_{n\in B\setminus A_0}$ which is divergent to infinity. Hence, $y\in AOS(\I^+)$.
\end{proof}


\subsection{Spaceability}
\label{subsec:spaceability}

For any $x\in \ell_1$, we write 
$$\lVert x\rVert = \sum_{n=1}^\infty |x(n)| \text{\ \ and \ \ } \nosnik(x) = \{n\in \N:x(n)\neq 0\}.$$

For an ideal $\I$ and a set $C\in \I^+$, we define an ideal $\I\restriction C$ on the set $C$ by 
$$\I\restriction C = \{A\cap C:A\in \I\}.$$

\begin{theorem}
\label{thm:spaceability-1}
Let $\I$ be an ideal on $\N$ such that 
$\I\restriction C$ is not a maximal ideal for every $C\in \I^+$ (i.e.~every set from $\I^+$ can be divided into two disjoint sets from $\I^+$). 
Then the following conditions are equivalent.
\begin{enumerate}
    \item $AOS(\I)\neq\emptyset$.
    \item $AOS(\I^*)\neq\emptyset$.
    \item $AOS(\I)$ is $\mathfrak{c}$-lineable.
    \item $AOS(\I^*)$ is $\mathfrak{c}$-lineable.
    \item $AOS(\I)$ is spaceable.
    \item $AOS(\I^*)$ is spaceable.
    \item $\I_{1/n}\cap \I^+\neq \emptyset$.
\end{enumerate}
\end{theorem}

\begin{proof}
The equivalence of (1), (2), (3), (4) and (7) is due to Theorem~\ref{the:lineability}.
It is known (\cite[Theorem I-1]{MR12204}, see also \cite{MR320717} or \cite{MR1825451}) that every infinite-dimensional Banach space has dimension at least $\continuum$, so we obtain the implications (6)$\implies$(4) and (5)$\implies$(3). 
By $AOS(\I)\subseteq AOS(\I^*)$ we get (5)$\implies$(6). 
Thus, it suffices to show (1)$\implies$(5).

Let $(b_n)\in AOS(\I)$.
Without loss of generality, we can assume that $\lVert (b_n)\rVert=1$ and $b_n\geq 0$ for each $n\in \N$. 
Then there exists $\varepsilon>0$ such that 
$$C = \{n\in \N: nb_n\geq \varepsilon\}\notin\I,$$
and, consequently, there exist 
pairwise disjoint sets $D_n\notin \I$  such that $D_n\subseteq C$ for each $n\in \N$.

For each $i,n\in \N$, we define 
$$x^{(i)}(n) =\begin{cases}
b_n&\text{if $n\in D_i\setminus\{\min D_i\}$,}\\
\displaystyle 1-\sum_{n\in D_i\setminus\{\min D_i\}}b_n&\text{if $n=\min D_i$,}\\
0&\text{otherwise.}
\end{cases}$$
Then  $\lVert x^{(i)}\rVert=1$, $\nosnik(x^{(i)}) = D_i$ and $\nosnik(x^{(i)})\cap \nosnik(x^{(j)})=\emptyset$ for each $i,j\in \N$, $i\neq j$.
Thus  
$$V=\left\{\sum_{i=1}^\infty t_i x^{(i)}:\ (t_i)\in\ell_1\right\}$$ 
is a closed subspace of infinite dimension in $\ell_1$. Hence, we only need to show that $V\subseteq AOS(\I)\cup\{0\}$. 

Let $(t_i)\in\ell_1$. If $t_i=0$ for all $i\in\N$ then obviously $\sum_{i=1}^\infty t_i x^{(i)}\in AOS(\I)\cup\{0\}$. 
Suppose that $t_{i_0}\neq 0$ for some $i_0\in \N$.
Then for any $n\in D_{i_0}\setminus\{\min D_{i_0}\} $  we have
\begin{equation*}
    \begin{split}
\left|n  \left(\sum_{i=1}^\infty t_i x^{(i)}(n)\right)\right| 
&=
\left|n  t_{i_0}  b_n\right|
\geq 
\left| t_{i_0} \varepsilon\right|>0.
    \end{split}
\end{equation*}
Since 
$D_{i_0} \notin\I$,
we obtain that the sequence 
$\left(n \left(\sum_{i=1}^\infty t_i x^{(i)}(n)\right)\right)_n$
is not $\I$-convergent to zero, hence it belongs to $AOS(\I)$.
\end{proof}

By identifying sets of natural numbers with their characteristic functions,
we equip $\cP(\N)$ with the topology of the Cantor space $\{0,1\}^\N$ (with the product measure of a countable sequence of the uniform measures on each $\{0,1\}$, resp.) and therefore
we can assign topological (measure-theoretic) notations to ideals on $\N$.
In particular, an ideal $\I$ has the \emph{Baire property} (is \emph{Lebesgue measurable} or is \emph{Borel}, resp.) if $\I$ has the Baire property  (is Lebesgue measurable or is a Borel set, resp.) as a subset of $\{0,1\}^\N$. 
For instance, summable ideals are Borel (even $F_\sigma$) ideals.


We say that an ideal $\I$ has the \emph{hereditary Baire property} (is \emph{hereditary Lebesgue measurable}, resp.) if $\I\restriction C$ has the Baire property (is Lebesgue measurable, resp.) for every $C\in \I^+$. (Using \cite[Proposition~2.1]{MR3624786}, one can construct an ideal with the Baire property which does not have the hereditary Baire property).
On the other hand, there is no use defining the hereditary Borel ideals because it is known that if $\I$ is a Borel ideal then $\I\restriction C$ is a Borel ideal for every $C\in\I^+$ (see for instance the proof of \cite[Theorem~3.13]{MR4358610}).
Consequently, Borel ideals have the hereditary Baire property as well as they are hereditary Lebesgue'a measurable.

\begin{theorem}
\label{thm:spaceability-1-Borel}
Let $\I$ be an ideal on $\N$ which has the hereditary Baire property or is hereditary Lebesgue measurable (in particular, if it is a Borel ideal).
Then the following conditions are equivalent.
\begin{enumerate}
    \item $AOS(\I)\neq\emptyset$.
    \item $AOS(\I^*)\neq\emptyset$.
    \item $AOS(\I)$ is $\mathfrak{c}$-lineable.
    \item $AOS(\I^*)$ is $\mathfrak{c}$-lineable.
    \item $AOS(\I)$ is spaceable.
    \item $AOS(\I^*)$ is spaceable.
    \item $\I_{1/n}\cap \I^+\neq \emptyset$.
\end{enumerate}
\end{theorem}

\begin{proof}
Let $\I$ be an ideal with the hereditary Baire property (hereditary Lebesgue measurable, resp.).
Since a maximal ideal does not have the Baire property and is not Lebesgue measurable (see e.g.~\cite[{Theorem~4.1.1}]{MR1350295}), we obtain that $\I\restriction C$
is not a maximal ideal  for any $C\in \I^+$.
Thus, Theorem~\ref{thm:spaceability-1} finishes the proof.
\end{proof}

An ideal $\I$ on $\N$ is called \emph{tall} if for every infinite $A\subseteq\N$ there is an infinite $B\subseteq A$ such that $B\in \I$.

The assumption of Theorem~\ref{thm:spaceability-1} is not satisfied for some non-tall ideals. Below, we provide an additional result (see Theorem~\ref{thm:spaceability-2}) which for instance guarantees that $AOS(\I)$ is spaceable for every non-tall ideal (see Corollary~\ref{cor:spaceability-2-non-tall}).

By $e_D:\N\to D$  we denote the increasing enumeration of a set $D\subseteq \N$.

\begin{theorem}
\label{thm:spaceability-2}
If  $\I$ is  an ideal on $\N$ such that 
there exist pairwise disjoint sets $D_n\in \I^+$, $n\in \N$, and a set $C\in \I^+\cap \I_{1/n}$ such that 
$$\{e_{D_n}(i):i\in C\}\in\I^+$$
for each $n\in \N$, 
then 
\begin{enumerate}
    \item $AOS(\I)\neq\emptyset$,
    \item $AOS(\I^*)\neq\emptyset$,
    \item $AOS(\I)$ is $\mathfrak{c}$-lineable,
    \item $AOS(\I^*)$ is $\mathfrak{c}$-lineable,
    \item $AOS(\I)$ is spaceable,
    \item $AOS(\I^*)$ is spaceable.
\end{enumerate}
\end{theorem}

\begin{proof}
Since $\I^+\cap \I_{1/n}\neq\emptyset$, we obtain 
(1), (2), (3) and (4) from Theorem~\ref{the:lineability}.
By $AOS(\I)\subseteq AOS(\I^*)$ we get (5)$\implies$(6). 
Thus, it suffices to show (5).

Let $D_n$ and $C$ be as in the assumption of the theorem.

For each $n\in \N$ we define
$$a_n = \begin{cases}
\frac{1}{n}&\text{for $n\in C$},\\
\frac{1}{n^2}&\text{otherwise}.
\end{cases}$$
Then  $(a_n)\in \ell_1$, $\lVert (a_n)\lVert>0$
and the sequence $(na_n)$ is not $\I$-convergent to zero.
Now, we define $b_n=a_n/\lVert (a_n)\rVert$ for each $n\in \N$, and notice that 
$(b_n)\in \ell_1$, $\lVert (b_n)\rVert=1$ and the sequence $(nb_n)$ is not $\I$-convergent to zero.

For each $i,n\in \N$, we define 
$$x^{(i)}(n) =\begin{cases}
b_j&\text{if $n\in D_i$, $n=e_{D_i}(j)$},\\
0&\text{otherwise.}
\end{cases}$$
Then  $x^{(i)}\in \ell_1$, $\lVert x^{(i)}\rVert=1$, $\nosnik(x^{(i)}) =  D_i$ and $\nosnik(x^{(i)})\cap \nosnik(x^{(j)})=\emptyset$ for each $i,j\in \N$, $i\neq j$.
Thus  
$$V=\left\{\sum_{i=1}^\infty t_i x^{(i)}:\ (t_i)\in\ell_1\right\}$$ 
is a closed subspace of infinite dimension. Hence, we only need to show that $V\subseteq AOS(\I)\cup\{0\}$. 

Let $(t_i)\in\ell_1$. If $t_i=0$ for all $i\in\N$ then obviously $\sum_{i=1}^\infty t_i x^{(i)}\in AOS(\I)\cup\{0\}$. 
Suppose that $t_{i_0}\neq 0$ for some $i_0\in \N$.
Then for any $j\in C $ and $n=e_{D_{i_0}}(j)$ we have
\begin{equation*}
    \begin{split}
\left|n\left(\sum_{i=1}^\infty t_i x^{(i)}(n)\right)\right| 
&=
\left|e_{D_{i_0}}(j)  \left(\sum_{i=1}^\infty t_i x^{(i)}(e_{D_{i_0}}(j))\right)\right|
\\&=
\left|e_{D_{i_0}}(j)  t_{i_0} x^{(i_0)}(e_{D_{i_0}}(j))\right|
\\&=
\left|e_{D_{i_0}}(j)  t_{i_0}  b_j\right|
\\&\geq 
\left|j  t_{i_0}  b_j\right|
=
\left| j t_{i_0} \cdot \frac{1/j}{\lVert (a_k)\rVert}\right| 
= 
\frac{|t_{i_0}|}{\lVert (a_k)\rVert} >0. 
    \end{split}
\end{equation*}
Since 
$\{e_{D_{i_0}}(j):j\in C\} \notin\I$,
we obtain that the sequence 
$\left(n \left(\sum_{i=1}^\infty t_i x^{(i)}(n)\right)\right)_n$
is not $\I$-convergent to zero, hence it belongs to $AOS(\I)$.
\end{proof}

\begin{corollary}
\label{cor:spaceability-2-non-tall}
If an ideal $\I$ is not tall, then 
\begin{enumerate}
    \item $AOS(\I)\neq\emptyset$,
    \item $AOS(\I^*)\neq\emptyset$,
    \item $AOS(\I)$ is $\mathfrak{c}$-lineable,
    \item $AOS(\I^*)$ is $\mathfrak{c}$-lineable,
    \item $AOS(\I)$ is spaceable,
    \item $AOS(\I^*)$ is spaceable.
\end{enumerate}
 \end{corollary}

\begin{proof}
Let  $A\subseteq\N$ be an infinite set which does not contain an infinite subsets from $\I$.
Let $D_n\subseteq A$, $n\in \N$, be pairwise disjoint infinite sets.
Take any  $C\in \I^+\cap \I_{1/n}$.
Then $C$ is infinite, so 
$\{e_{D_n}(i):i\in C\}$ is an infinite subset of $A$, hence it belongs to $\I^+$. 
Now, Theorem~\ref{thm:spaceability-2} finishes the proof.
\end{proof}

\begin{corollary}
\label{cor:spaceability-2}
Let $\I$ be an ideal on $\N$ such that 
\begin{itemize}
\item 
there exists  an infinite  partition of $\N$ into sets from $\I^+$, 
\item 
for each $B\in \I^+$ there exists $D\subseteq B$ such that $D\in \I^+$ and
$$\forall A\subseteq \N \, \left(A\in \I\iff  \{e_D(i):i\in A\}\in \I\right)$$
(i.e.~the bijection $e_D$ witnesses the fact that the ideals $\I$ and $\I\restriction D$ are isomorphic).
\end{itemize}
Then the following conditions are equivalent.
\begin{enumerate}
    \item $AOS(\I)\neq\emptyset$.
    \item $AOS(\I^*)\neq\emptyset$.
    \item $AOS(\I)$ is $\mathfrak{c}$-lineable.
    \item $AOS(\I^*)$ is $\mathfrak{c}$-lineable.
    \item $AOS(\I)$ is spaceable.
    \item $AOS(\I^*)$ is spaceable.
    \item $\I_{1/n}\cap \I^+\neq \emptyset$.
\end{enumerate}
\end{corollary}

\begin{proof}
The equivalence of (1), (2), (3), (4) and (7) is due to Theorem~\ref{the:lineability}.
It is known (\cite[Theorem I-1]{MR12204}, see also \cite{MR320717} or \cite{MR1825451}) that every infinite-dimensional Banach space has dimension at least $\continuum$, so we obtain the implications (6)$\implies$(4) and (5)$\implies$(3). 
By $AOS(\I)\subseteq AOS(\I^*)$ we get (5)$\implies$(6). 
Thus, it suffices to show (7)$\implies$(5).

Let $C\in \I^+\cap \I_{1/n}$.
Let $B_n\in \I^+$, $n\in\N$, be an infinite  partition of $\N$.

For every $n\in \N$, we take $D_n\subseteq B_n$ such that $D_n\in \I^+$ and $e_{D_n}$ witnesses the fact that $\I$ and $\I\restriction D_n$ are isomorphic. 

Since $C\in \I^+$, we obtain that the set  $\{e_{D_n}(i):i\in C\}\in \I^+$ for each $n\in \N$. 

Now Theorem~\ref{thm:spaceability-2} finishes the proof.
\end{proof}

The first assumption of Corollary~\ref{cor:spaceability-2} can be characterized in terms of maximal ideals (see Proposition~\ref{prop:finite-intersection-of-ultrafilters}), which in turn can be  used  to show (see Proposition~\ref{prop:infinite-partition-of-positive-sets-BP-LEB})  that this assumption is valid for most ideals used in the literature (e.g.~for all Borel ideals).

\begin{proposition}[{\cite[Lemma~1.3]{MR3226022}}]
\label{prop:finite-intersection-of-ultrafilters}
Let $\I$ be an ideal on $\N$. Then the following conditions are equivalent.
\begin{enumerate}
\item There exists an infinite  partition of $\N$ into sets from $\I^+$. 
\item $\I$ is not equal to the intersection of finitely many maximal ideals.
\end{enumerate}
\end{proposition}

\begin{proposition}
\label{prop:infinite-partition-of-positive-sets-BP-LEB}
Let $\I$ be an ideal on $\N$.
If $\I$  has the Baire property or is Lebesgue measurable (in particular, if it is a Borel ideal), then there exists  an infinite  partition of $\N$ into sets from $\I^+$.
\end{proposition}

\begin{proof}
Let $\I$ be an ideal with the Baire property (Lebesgue measurable, resp.).
In view of Proposition~\ref{prop:finite-intersection-of-ultrafilters}, we only need to show that $\I$ is not the intersection of finitely many maximal ideals.
In \cite[Proposition~23]{MR579439} (\cite[Proposition 7]{MR579439}, resp), the authors proved that the intersection of countably many ideals without the Baire property (Lebesgue nonmeasurable, resp.) does not have the Baire property (is not Lebesgue measurable, resp.). 
Since maximal ideals do not have the Baire property and are not Lebesgue measurable (see e.g.~\cite[{Theorem~4.1.1}]{MR1350295}), we obtain that $\I$ is not the intersection of countably  many (in particular, finitely many)  maximal ideals.
\end{proof}

Below, we show two examples of tall ideals which satisfy assumptions of Corollary~\ref{cor:spaceability-2}. 

\begin{example}[Hindman ideal]
A set $A\subseteq \N$ is an \emph{IP-set}  if there exists an infinite set $D\subseteq\N$ such that $\FS(D)\subseteq A$ where $\FS(D)$ denotes the set of all finite non-empty sums of distinct elements of $D$.
It follows from Hindman's theorem (\cite{MR349574}, see also \cite{MR1044995}) that if $A\cup B$ is an IP-set, then $A$ or $B$ is an IP-set as well.
Thus, the family 
$$\I_{IP} = \{A\subseteq\N: \text{$A$ is not an IP-set}\}$$ 
is an ideal on $\N$.
The ideal $\I_{IP}$  is coanalytic as
$\I_{IP}^+$ is a projection on the first coordinate of a closed set 
$$B = \{(A,D)\in \{0,1\}^\N\times [\N]^\omega: \FS(D)\subseteq A\},$$
where $[\N]^\omega$ is a set of all infinite subsets of $\N$ (which is a $G_\delta$ subset of $\{0,1\}^\N$, hence the Polish space).
Consequently, $\I_{IP}$ has  the Baire property, so by Proposition~\ref{prop:infinite-partition-of-positive-sets-BP-LEB}, the ideal $\I_{IP}$ satisfies the first assumption of Corollary~\ref{cor:spaceability-2}.
The fact that $\I_{IP}$ satisfies the second assumption of Corollary~\ref{cor:spaceability-2} follows from the proof  of
\cite[Theorem~4.5]{MR3276758}.
\end{example}

\begin{example}[van der Waerden ideal]
A set $A\subseteq \N$ is an \emph{AP-set} if $A$ contains arithmetic progressions of arbitrary finite length. 
It follows from  van der Waerden's theorem (\cite{vanderWearden}, see also \cite{MR1044995}) that if $A\cup B$ is an AP-set, then $A$ or $B$ is an AP-set as well.
Thus, the family 
$$\I_{AP}  = \{A\subseteq\N: \text{$A$ is not an AP-set}\}$$ 
is an ideal on $\N$.
One can show, that this is a Borel ideal (even $F_\sigma$ ideal).
Indeed, $\I_{AP}$ is $F_\sigma$ as $\I_{AP}^+$ is $G_\delta$ and that is true because 
$\I_{AP}^+=\bigcap_{n=1}^\infty\bigcup_{k=1}^\infty\bigcup_{r=1}^\infty A_{n,k,r},$
where $A_{n,k,r}=\{A\subseteq\N:\{k,k+r,k+2r,\ldots, k+nr\}\subseteq A\}$ is a basic open set.
Therefore, by Proposition~\ref{prop:infinite-partition-of-positive-sets-BP-LEB}, the ideal $\I_{AP}$ satisfies the first assumption of Corollary~\ref{cor:spaceability-2}.
The fact that $\I_{AP}$ satisfies the second assumption of Corollary~\ref{cor:spaceability-2} follows from the proof  of
\cite[Theorem~3.3]{MR3276758}.
\end{example}

\begin{theorem}
If  $\I$ is  an ideal on $\N$ such that 
there exists  a set $C\in \I^+\cap \I_{1/n}$ such that  there exists an infinite partition of $C$ into sets from $\I^+$,  
then 
\begin{enumerate}
    \item $AOS(\I)\neq\emptyset$,
    \item $AOS(\I^*)\neq\emptyset$,
    \item $AOS(\I)$ is $\mathfrak{c}$-lineable,
    \item $AOS(\I^*)$ is $\mathfrak{c}$-lineable,
    \item $AOS(\I)$ is spaceable,
    \item $AOS(\I^*)$ is spaceable.
\end{enumerate}
\end{theorem}

\begin{proof}
Let $(b_n)$ be defined as in the proof of Theorem~\ref{thm:spaceability-2} and then proceed as in  the proof of Theorem~\ref{thm:spaceability-1}.
\end{proof}

Let $C\in \I_{1/n}$ be an infinite set. Let $\I$ be an ideal  such that 
$\I\restriction C$ is a Borel ideal and $\I\restriction (\N\setminus C)$ is a maximal ideal.
Then  the ideal  $\I$
satisfies the assumption of the above theorem, but it does not satisfy the assumption of Theorem~\ref{thm:spaceability-1}.

\begin{question}
\label{q:spaceable-iff}
Is $AOS(\I)\neq\emptyset$ ($AOS(\I^*)\neq\emptyset$,  resp.) a necessary and sufficient condition for  $AOS(\I)$ ($AOS(\I^*)$, resp.) to be spaceable for each ideal $\I$?
\end{question}

\begin{question}
\label{q:nontrivial-non-spaceable-ideal}
Does there exist an ideal $\I$ such that $AOS(\I^+)\neq\emptyset$ is  spaceable?
\end{question}


\subsection{Algebrability}
\label{subsec:algebrability}

\begin{theorem}
\label{thm:strong-algebrability}
Let  $(a_k)$ be a~sequence tending to infinity and $(m_k)$ be an increasing sequence of positive integers such that 
$$m_k\geq k^{a_k} \text{\ \ \ for each $k$}.$$
If $M=\{m_k: k\in \N\}$ and 
\begin{enumerate}
    \item 
 $M \in\I^+$, then $AOS(\I)$ and $AOS(\I^*)$  are strongly $\continuum$-algebrable;
    \item 
 $M\in\I^*$, then $AOS(\I^+)$ is strongly $\continuum$-algebrable.
\end{enumerate}\end{theorem}
\begin{proof}
Since $AOS(\I)\subseteq AOS(\I^*)$, strong $\continuum$-algebrability of $AOS(\I^*)$ will follow from
strong $\continuum$-algebrability of $AOS(\I)$.
Below we simultaneously prove  strong $\continuum$-algebrability of $AOS(\I)$ and $AOS(\I^+)$. 

Let  $(a_k)$ and  $M\in \I^+$ ($M\in \I^*$, resp.) satisfy the assumptions of the theorem.

Let $\Lambda\subseteq(1,2)$ be a linearly independent set over rationals with $|\Lambda|=\continuum$. 

For every $\alpha\in\Lambda$, we define a  sequence $a^{(\alpha)}$ by
$$a^{(\alpha)}(n)=\begin{cases}
\frac{1}{k^\alpha} & \text{ for } n=m_k\in M,\\
\frac{1}{n^\alpha} & \text{ for } n\not\in M. \end{cases}$$

Note that $a^{(\alpha)}\in \ell_1$ because 
$$\sum_{n=1}^{\infty}a^{(\alpha)}(n)\leq \sum_{k=1}^{\infty}\frac{1}{k^\alpha} +\sum_{n=1}^{\infty}\frac{1}{n^\alpha}<\infty.$$
Moreover, $a^{(\alpha)}\in AOS(\I)$ ($a^{(\alpha)}\in AOS(\I^+)$, resp.) as for every $n=m_k\in M$ we have
$$n a^{(\alpha)}(n)= m_k  a^{(\alpha)}(m_k)=\frac{m_k}{k^{\alpha}}\geq k^{a_k-\alpha}\geq k^{a_k-2},$$
which tends to infinity.

Using \cite[Fact~1.2]{MR3491544}, we know that 
in order to show strong $\continuum$-algebrability of $AOS(\I)$ ($AOS(\I^+)$, resp.), 
it is enough  to prove that 
$$P(a^{(\alpha_1)},\dots,a^{(\alpha_q)})\in AOS(\I) \text{\ $(\in AOS(\I^+)$, resp.)}$$  
for any pairwise distinct $\alpha_1,\ldots,\alpha_q\in\Lambda$
and any polynomial 
$$P(x_1,\ldots,x_q)=\sum_{i=1}^p c_i x_1^{\beta_{i,1}}\ldots x_q^{\beta_{i,q}},$$ 
where $c_i$ are nonzero reals and $[\beta_{i,j}]$ is a matrix of nonnegative integers with pairwise distinct, nonzero rows. 

First, observe that for any $m_k\in M$ we have
$$P(a^{(\alpha_1)},\dots,a^{(\alpha_q)}) (m_k)=\sum_{i=1}^p c_i k^{-(\alpha_1 \beta_{i,1}+\ldots+\alpha_q\beta{i,q})}=\sum_{i=1}^p c_ik^{-r_i}, $$
where $r_i=\alpha_1 \beta_{i,1}+\ldots+\alpha_q\beta{i,q}$.

Since $\Lambda$ is linearly independent, all $r_i$ are positive and pairwise distinct. We may assume that $r_1=\min\{r_1,\ldots,r_p\}$. 
Then 
\begin{equation*}
    \begin{split}
\left|m_kP(a^{(\alpha_1)},\ldots,a^{(\alpha_q)}) (m_k)\right| 
&= 
m_k\left|\sum_{i=1}^p c_ik^{-r_i}\right|
\\&= 
m_k \left|c_1k^{-r_1}\right|\cdot \left|1+ \sum_{i=2}^p \frac{c_i}{c_1} \cdot k^{-r_i+r_1}\right|
\\&\geq  
k^{a_k}\left|c_1k^{-r_1}\right|\cdot \left|1+ \sum_{i=2}^p \frac{c_i}{c_1} \cdot k^{-r_i+r_1}\right|
\\&=
\left|c_1\right|\cdot k^{a_k-r_1}\cdot \left|1+ \sum_{i=2}^p \frac{c_i}{c_1} \cdot \frac{1}{k^{r_i-r_1}}\right|
\xrightarrow{k\to\infty}\infty,
    \end{split}
\end{equation*}
because $r_i-r_1>0$ for each $i\geq 2$ and 
$a_k$ tends to infinity. 

Since $M \in\I^+$ ($M\in \I^*$, resp.),  we conclude that $P(a^{\alpha_1},\ldots,a^{\alpha_q})\in AOS(\I)$ ($\in AOS(\I^+)$, resp.).
\end{proof}

The ideal $\I = \{A\subseteq \N: A\cap\{n^n:n\in \N\} \text{ is finite}\}$
satisfies the assumptions of the above theorem, so 
$AOS(\I)$, $AOS(\I^*)$  
and $AOS(\I^+)$ 
are strongly $\continuum$-algebrable. However, 
the nonemptiness of 
$AOS(\I)$, $AOS(\I^*)$  
or $AOS(\I^+)$ does not guarantee even $1$-algebrability of these sets.

\begin{proposition}
\label{prop:necessary-not-sufficient-algebra}
There exists an ideal $\I$ such that $AOS(\I)\neq\emptyset$,  $AOS(\I^+)\neq\emptyset$ and $AOS(\I^*)\neq\emptyset$ but neither $AOS(\I)$ nor $AOS(\I^+)$ nor $AOS(\I^*)$ is 1-algebrable.
\end{proposition}

\begin{proof}
For a set $B=\{n^2:n\in\N\}$, we define a summable ideal 
$$\I = \left\{A\subseteq\N:\sum_{n\in A\cap B} \frac{1}{\sqrt{n}}<\infty\right\}.$$

Since  $B\in\I^*\cap\I_{1/n}$, the sets  $AOS(\I^+)$, $AOS(\I)$ and $AOS(\I^*)$ are  nonempty by Proposition~\ref{prop:AOSs-nonempty}.

Now, we show that $AOS(\I^*)$ does not contain any  subalgebra generated by a singleton. This will finish the proof as $AOS(\I)\subseteq AOS(\I^*)$ and $AOS(\I^+)\subseteq AOS(\I^*)$.

Take any $a=(a_n)\in AOS(\I^*)$. 

Then $C=\{n\in\N: |a_n|\leq 1/ \sqrt{n}\}\in\I^*$ as otherwise $B\setminus C\in \I^+$ and consequently 
$$\sum_{n\in\N} |a_n|\geq\sum_{n\in B\setminus C } \frac{1}{\sqrt{n}}=\infty, $$
a contradiction with $a_n\in \ell_1$. 

Now, consider the polynomial $P(x)=x^3$. Then for any $n\in C$ we have
$$|n a_n^3|\leq\frac{n}{(\sqrt{n})^3}=\frac{1}{\sqrt{n}}, $$
which tends to zero, thus $P(a_n)\not\in AOS(\I^*)$. 
\end{proof}


\bibliographystyle{amsplain}
\bibliography{paper}

\end{document}